\documentclass[a4paper,12pt]{article}
\usepackage{amssymb, amsmath, afterpage}
\usepackage{amsfonts,amsthm}
\newcommand{\+}[1]{\ensuremath{\overset{+}{#1}}}
\newcommand{\pli}[1]{\ensuremath{\overset{+}{\imath}}}
\newcommand{\mi}[1]{\ensuremath{\overset{-}{#1}}}
\newcommand{\arb}[1]{\ensuremath{\overset{\pm}{#1}}}

\newcommand{\sym}{\ensuremath{\mathrm{Sym}}}

\newcommand{\droot}{\ensuremath{\Phi_{\mathrm{long}}}}
\newcommand{\short}{\ensuremath{\Phi_{\mathrm{short}}}}
\newcommand{\supp}{\ensuremath{\mathrm{supp}}}

\newtheorem{thm}{Theorem}[section]
\newtheorem{lemma}[thm]{Lemma}         
\newtheorem{prop}[thm]{Proposition}

\newtheorem{defn}[thm]{Definition}
\newtheorem{cor}[thm]{Corollary}
\begin{document}
    \title{\Large\textbf{Maximal Length Elements of Excess Zero in Finite Coxeter Groups}}
\date{}
      \author{S.B. Hart and P.J. Rowley\thanks{The authors wish to acknowledge support for this work from
      a London Mathematical Society Research in Pairs Grant.}}
  \maketitle
\vspace*{-10mm}

\section{Introduction}

Conjugacy classes of finite Coxeter groups have long been of
interest, the correspondence between partitions and conjugacy
classes of the symmetric groups having been observed by Cauchy
\cite{cauchy} in the early days of group theory. For Coxeter groups
of type $B_n$ and $D_n$, descriptions of their conjugacy classes, by
Specht \cite{specht} and Young \cite{young}, have also been known
for a long time. In 1972, Carter \cite{carter} gave a uniform and
systematic treatment of the conjugacy classes of Weyl groups. More
recently, Geck and Pfeiffer \cite{gkpf} reworked Carter's
description from more of an algorithmic standpoint. Motivation for
investigating the conjugacy classes of finite Coxeter groups, and
principally those of the irreducible finite Coxeter groups, has come
from many directions, for example in the representation theory of
these groups and the classification of maximal tori in groups of Lie
type (see \cite{carter2}). The behaviour of length in a conjugacy
class is frequently important. Of particular interest are those
elements of minimal and maximal lengths in their class. Instrumental to Carter's work was establishing the fact that in a finite Coxeter group every element is either an involution or a product of two involutions. Given the importance of the length function, it is natural to ask whether for an element $w$ it is possible to choose two involutions $\sigma$ and $\tau$ with $w = \sigma\tau$
in such a way that combining a reduced expression for $\sigma$ with one for $\tau$ produces a reduced expression for $w$. That is, can we ensure that the length $\ell(w)$ is given by $\ell(w) = \ell(\sigma) + \ell(\tau)$? Not surprisingly, the answer to this is, in general, no. This naturally leads to introducing the concept of {\em excess} of $w$, denoted by $e(w)$, and defined by
$$e(w) = \min\{\ell(\sigma) + \ell(\tau) - \ell(w): \sigma\tau = w, \sigma^2 = \tau^2 = 1\}.$$ In \cite{invp}, \cite{zeroex} and \cite{onxs}, various properties of excess were investigated. It was shown, among other things, that in every conjugacy class of a finite Coxeter group $W$ there is an element $w$ of minimal length in the conjugacy class, such that the excess of $w$ is zero \cite[Theorem 1.1]{zeroex}. This raises the question as to whether there is also an element of maximal length and excess zero.\\

In this paper we address this question and show that elements of maximal length and excess zero do indeed exist.

\begin{thm}
\label{main} Let $W$ be a finite Coxeter group and $C$ a conjugacy class of $W$. Then there exists an element $w$ of maximal length in $C$ such that $e(w) = 0$.
\end{thm}

Theorem \ref{main}, of course, prompts the question as to how many maximal elements of a conjugacy class have excess zero. Intriguingly, calculations in classical Weyl groups up to rank 10 reveal that every element of maximal length in a conjugacy class has excess zero. Equally intriguing is the fact that $E_6$ has, out of its 25 conjugacy classes, two conjugacy classes for which this is not the case, while $E_7$ has exactly one conjugacy class out of 60 conjugacy classes for which not all maximal length elements have excess zero. See Section 5 for further discussion of these cases. We remark that, in stark contrast, there are many instances where not all minimal length elements in a conjugacy class have excess zero (see \cite{zeroex}).\\

In the course of proving Theorem \ref{main},  we need a workable description of representatives of maximal length in conjugacy classes of Coxeter groups of types $A_n$, $B_n$ and $D_n$.  Minimal
length elements in conjugacy classes of Coxeter groups have received considerable attention -- see
\cite{gkpf}. Now every
finite Coxeter group $W$ possesses a (unique) element
$w_0$ of maximal length in $W$. For $C$ a conjugacy class of $W$,
set $C_0 = Cw_0 = \{ww_0: w \in C\}$. If, as happens in many
cases, $w_0 \in Z(W)$, then $C_0$ is also a conjugacy class of
$W$. Moreover, $w \in C$ has minimal length in $C$ if and only if
$ww_0$ has maximal length in $C_0$. Thus information about
maximal length elements in a conjugacy class may be obtained from
that known about minimal length elements. Among the finite
irreducible Coxeter groups, only those of type $I_{m}$ ($m$ odd),
$A_n$, $D_n$ ($n$ odd) and $E_6$ have $w_0 \notin Z(W)$.
The first of these, being just dihedral groups, are quickly dealt
with. A description of certain maximal length elements in conjugacy classes
of type $A_n$ were given by Kim \cite{kim} and for $E_6$ see Table
III of \cite{twisted}. In Section \ref{bdl} of this paper we deal with type $D_n$ (and in doing so give a result for type $B_n$ at the same time). Representatives of maximal length
for type $D_n$ could be extracted from Section 4 of \cite{twisted}, but here we give a more direct treatment that deals with both type $B_n$ and type $D_n$ and gives more information about the number of long and short roots taken negative by elements of maximal length. Theorem \ref{thm1} gives an expression for the maximal length of elements in a given conjugacy class for type $D_n$ while Theorem \ref{thm2} below gives a list of maximal length class representatives in types $B_n$ and $D_n$. This is what we require for our work on elements of excess zero.\\

Theorems \ref{thm2} and \ref{thm1} are consequences of a more general result, Theorem \ref{maxl}, concerning $D$-lengths and $B$-lengths of elements in a Coxeter group $W$ of type $B_n$ ($D$-length and $B$-length will be defined in Section \ref{bdl}). Suppose $\hat W$ is of type $D_n$. Then we may regard $\hat W$ as a canonical index 2 subgroup of $W$ where $W$ is a Coxeter group of type $B_n$. Let $C$ be a conjugacy class of $W$ that is contained in $\hat W$. In the case when $n$ is odd, $w_0 \neq \hat w_0$ (the longest element of $\hat W$) and consequently $C_0 = Cw_0$ is not even a subset of $\hat W$, much less a conjugacy class of $\hat W$. However, working in the wider context of $W$, we are able
to obtain elements of maximal $D$-length in $C$ from suitable elements of minimal $B$-length in $C_0$. Therefore, in the course of establishing Theorem \ref{thm1}, we also produce representative elements of maximal length in their conjugacy class. To describe these elements, recall that conjugacy classes in types $B_n$ and $D_n$ are parameterized by signed cycle type (this will be described fully in Section \ref{bdl}), with some classes splitting in type $D_n$.\\

Let $\lambda = (\lambda_1, \ldots, \lambda_m)$ be a partition of $n$, and let $\rho \geq 0$. For ease of notation set $\mu_i = \sum_{j=1}^{i-1} \lambda_j$ (and by convention $\mu_1 = 0$).
We then define the {\em corresponding signed element} $w_{\lambda,\rho}$ to be
$w_{\lambda,\rho} = w_1 \cdots w_m$ where $$w_i = \left\{\begin{array}{ll}
(\mi{\mu_i + 1}, \mi{\mu_i + 2}, \ldots, \mi{\mu_{i+1}-1}, \mi{\mu_{i+1}})
 & \text{ if $1\leq i \leq \rho$;}\\
(\mi{\mu_i + 1}, \mi{\mu_i + 2}, \ldots, \mi{\mu_{i+1}-1}, \+{\mu_{i+1}}) & \text{ if $ \rho < i \leq m$.}
\end{array}\right.$$

We call $\lambda$ a {\em maximal split partition} (with respect to $\rho$) if $\lambda_1 \geq \cdots \geq \lambda_{\rho}$ and $\lambda_{\rho+1} \geq \cdots \geq \lambda_m$. For example, if $\lambda = (5,2,4,3)$ and $\rho = 2$, then $$w_{\lambda,\rho} = (\mi 1,\mi 2,\mi 3,\mi 4,\mi 5)(\mi 6,\mi 7)(\mi 8,\mi 9,\mi{10},\+{11})(\mi{12},\mi{13},\+{14}).$$

 Our second main result in this paper is the following.

\begin{thm}\label{thm2} Let $W$ be of type $B_n$ and $\hat W$ its canonical subgroup of type $D_n$.
Every conjugacy class of $W$ contains an element $w_{\lambda,\rho}$, where $\lambda$ is a maximal split partition with respect to $\rho$. Furthermore each element $w_{\lambda,\rho}$ has maximal $B$-length and maximal $D$-length in its conjugacy class of $W$.
\end{thm}

 Representatives of minimal length in conjugacy classes of types $B_n$ and $D_n$ appear in \linebreak Theorems 3.4.7 and 3.4.12 of \cite{gkpf}. However we need additional information about elements of minimal length in $W$-conjugacy classes, which gives as a byproduct (in Corollaries \ref{cor1} and \ref{cor2}) an alternative proof that the representatives given in \cite{gkpf} are indeed of minimal length.\smallskip\\
%

In the rest of this section we briefly discuss the proof of Theorem \ref{main}. Given a root system $\Phi$ for a Coxeter group $W$, we have that $\Phi$ is the disjoint union of the set of positive roots $\Phi^+$ and the set of negative roots $\Phi^- = -\Phi^+$. For details on root systems, including these observations, see for example Chapter 5 of \cite{humphreys}. It is well known (for example Proposition 5.6 of \cite{humphreys}) that for any $w$ in $W$, the length $\ell(w)$ is given by $$\ell(w) = |N(w)| = |\{\alpha \in \Phi^+ : w(\alpha) \in \Phi^-\}|.$$
That is, $\ell(w)$ is the number of positive roots taken negative by $w$. We emphasise here that, in line with other work on Coxeter groups, elements of the group will act on the left. It is easy to show that if $w = gh$ for some $g, h \in W$, then \begin{equation}
 \ell(w) = \ell(g) + \ell(h) - 2|N(g)\cap N(h^{-1})|.\label{leneq}
\end{equation}
(Equation \eqref{leneq} is well known but is stated and proved as part of Lemma 2.1 in \cite{zeroex}.) Our method of proving Theorem \ref{main} for the classical Weyl groups will be as follows. First we will establish a collection of elements $w$ constituting a representative of maximal length for each conjugacy class of the group under consideration. For each such $w$, we will obtain involutions $\sigma$ and $\tau$ such that $N(\sigma) \cap N(\tau) = \emptyset$ and $\sigma\tau = w$. It follows from Equation \eqref{leneq} that the excess of $w$ is zero.
We conclude this section with two lemmas which will be useful later.
\begin{lemma}
\label{prelim} Let $W$ be a Coxeter group.
Let $g, h \in W$ and suppose $N(g) \cap N(h^{-1}) = \emptyset$. Then $N(gh) = N(h) \dot{\cup} h^{-1}(N(g))$.
\end{lemma}
\begin{proof}
Note that $|N(h) \cap h^{-1}N(g)| = |hN(h) \cap N(g)| \leq |\Phi^{-} \cap N(g)| = 0$. So $N(h)$ and $h^{-1}(N(g))$ are indeed disjoint. Suppose $\alpha \in N(h)$. Then $gh(\alpha) \in \Phi^+$ would imply that $h(\alpha) \in -N(g)$, which implies $-h(\alpha) \in N(h^{-1}) \cap N(g)$, a contradiction. Hence $gh(\alpha) \in \Phi^-$, meaning $N(h) \subseteq N(gh)$. Now suppose $\alpha \in N(gh)\setminus N(h)$. Then $h(\alpha) \in \Phi^+$ but $gh(\alpha) \in \Phi^-$. Therefore $\alpha \in h^{-1}(N(g))$. Conversely, since $N(h^{-1}) \cap N(g) = \emptyset$, we have $h^{-1}(N(g)) \subseteq \Phi^+$ and so $h^{-1}(N(g)) \subseteq N(gh)$. Therefore $N(gh) = N(h) \dot{\cup} h^{-1}(N(g))$.
\end{proof}

\begin{lemma}
\label{useful} Let $W$ be a Coxeter group.
Suppose $t_1, t_2, \ldots, t_m$ are involutions with the property that whenever $i\neq j$ we have $t_i(N(t_j)) = N(t_j)$. Then  $N(t_1\cdots t_m) = \dot\cup_{i=1}^m N(t_i)$.
\end{lemma}
\begin{proof}
The result clearly holds when $m=1$. Assume the result holds for $m=k$. Set $u_k = t_1t_{2}\cdots t_k$. Then inductively $N(u_k) = \dot\cup_{i=1}^k N(t_i)$. If $\alpha \in N(u_k)$, then $\alpha \in N(t_i)$ for some $i \leq k$ and so $t_{k+1}(\alpha) \in N(t_i) \subseteq \Phi^+$. Thus $N(u_k) \cap N(t_{k+1}) = \emptyset$. Lemma~\ref{prelim} now gives $N(u_{k+1}) =  N(t_{k+1}) \dot\cup t_{k+1}\left(N(u_{k})\right) = \dot{\cup}_{i=1}^{k+1} N(t_i)$. The result follows by induction.
\end{proof}

Finally for $\sigma \in \sym(n)$, the {\em support} of $\sigma$, denoted $\supp(\sigma)$ is simply the set of points not fixed by $\sigma$. That is, $$\supp(\sigma) = \{i \in \{1, \ldots, n\} : \sigma(i) \neq i\}.$$

\section{Type $A_{n-1}$}\label{secan}

The permutation group $\sym(n)$ is a Coxeter group of type $A_{n-1}$. So throughout this section we will set $W = \sym(n)$. In this context then, the length of an element $w$ is the number of inversions, that is the number of pairs $(i,j)$ with $1\leq i < j \leq n$ such that $w(i) > w(j)$. We can also think of this in terms of the root system (which we can consider as a warm up for the type $B_n$ and $D_n$ cases). For the root system $\Phi$ we can take $$\Phi^+ = \{e_i - e_j: 1 \leq i < j \leq n\}$$ and $\Phi^- = -\Phi^+$. Hence $$N(w) = \{e_i - e_j: i < j, w(i) > w(j)\}.$$ For what follows it will sometimes be helpful to consider intervals   $[a,b]$ for $1\leq a < b\leq n$. The group $\sym([a,b])$ is a standard parabolic subgroup of $W$, and by $\Phi^+_{[a,b]}$ we mean $\{e_i - e_j: a \leq i < j \leq b\}$. We note that if $w \in \sym([a,b])$ then $N(w) \subseteq \Phi^+_{[a,b]}$.  The conjugacy classes of $W$ are parameterized by partitions of $n$. Kim \cite{kim} has described a set of representative elements of maximal length in conjugacy classes of $\sym(n)$, using the `stair form'. Following \cite{kim} we give the following definition.
\begin{defn} \label{defkim} Let $n$ be a positive integer.
\begin{enumerate}
\item[(i)] Define the sequence $a_1, a_2, \ldots, a_n$ by $a_{2i-1} = i$ and $a_{2i} = n - (i-1)$. (So $a_1 = 1$, $a_2 = n$, $a_3 = 2$, $a_4 = n-1$ and so on.)
\item[(ii)] Given a partition $\lambda = (\lambda_1,\ldots, \lambda_m)$ of $n$, its {\em corresponding element} is the element of $\sym(n)$ defined by $$w_{\lambda} = w_1w_{2}\cdots w_{m}$$ where $w_i = (a_{\lambda_1 + \cdots + \lambda_{i-1} + 1}, a_{\lambda_1 + \cdots + \lambda_{i-1} + 2}, \ldots, a_{\lambda_1 + \cdots + \lambda_{i-1} + \lambda_{i}})$.
\item[(iii)] Let $\lambda = (\lambda_1,\ldots, \lambda_m)$ be a partition of $n$. Then $\lambda$ is a {\em maximal partition} of $n$ if there exists $\ell$, with $0 \leq \ell \leq m$ such that $\lambda_1, \ldots, \lambda_{\ell}$ are even numbers in any order, and $\lambda_{\ell+1}, \ldots, \lambda_{m}$ are odd numbers in decreasing order. (In \cite{kim} this is referred to as a maximal composition.)
\end{enumerate}
\end{defn}

For example, given the maximal partition (4,5) of 9, the corresponding element of $\sym(9)$ is $(1, 9, 2, 8)(3, 7, 4, 6, 5)$.
Any partition of $n$ can be reordered so as to produce a maximal partition. Therefore each conjugacy class can be represented by a maximal partition.  We can now state the main result of \cite{kim}.

\begin{thm}
[Kim, \cite{kim}] \label{kimthm}
Let $\lambda = (\lambda_1,\ldots, \lambda_m)$ be a maximal partition of $n$. The corresponding element $w_\lambda$ of $\lambda$ has maximal length in its conjugacy class.
\end{thm}

Given a sequence $b_1, b_2, \ldots b_k$, of distinct elements in $\{1,\ldots, n\}$, we define $g_{b_1, \cdots, b_k}$ to be the permutation that reverses the sequence and fixes all other $c \in \{1,\ldots, n\}$, so that $g(b_i) = b_{k+1-i}$. That is, $$g = (b_1,b_k)(b_2,b_{k-1}) \cdots (b_{\lfloor k/2 \rfloor}, b_{\lceil k/2 \rceil + 1}).$$ In particular, $g_{b_1, \cdots, b_k}$ is an involution.\\

Let $w = (b_1, b_2, \cdots, b_k)$. Define
\begin{align}
\sigma(w) &=
\left\{\begin{array}{ll} g_{b_1, \ldots, b_k} & \text{ if $k$ even}\\ g_{b_2, \ldots, b_k} & \text{ if $k$ odd}
\end{array}\right. \label{sigmadef}\\
\tau(w) &= \left\{\begin{array}{ll} g_{b_1, \ldots, b_{k-1}} & \text{ if $k$ even}\\ g_{b_1, \ldots, b_k} & \text{ if $k$ odd}
\end{array}\right. \label{taudef}
\end{align}

\begin{lemma}\label{2.3}
Let $w$ be a cycle of $\sym(n)$. Then writing $\sigma = \sigma(w)$ and $\tau = \tau(w)$ we have that $w = \sigma \tau$, where $\sigma$ and $\tau$ are both involutions.
\end{lemma}

\begin{proof} It is clear that $\sigma$ and $\tau$ are involutions. Let $w = (b_1, \cdots, b_k)$. If $k$ is even, then by \eqref{sigmadef} and \eqref{taudef} we see that for $i \leq k-1$ we have $\sigma\tau(b_i) = \sigma(b_{k-i}) = b_{(k+1) - (k-i)} = b_{i+1}$, and $\sigma\tau(b_k) = \sigma(b_k) = b_1$. Therefore $w = \sigma\tau$. If $k$ is odd then $\sigma(b_j) = b_{k+2-j}$ when $2\leq j \leq k$, and $\sigma(b_1) = b_1$. Therefore, when $i \leq k-1$ we have $\sigma\tau(b_i) = \sigma(b_{k+1-i}) = b_{k+2-(k+1-i)} = b_{i+1}$ and $\sigma\tau(b_k) = \sigma(b_1) = b_1$. Again we get $w = \sigma\tau$.
\end{proof}

Before we go further we introduce some additional notation. Any partition $\lambda$ (via its corresponding element $w_{\lambda}$) induces a partition $X = (X_1, \ldots, X_m)$ of $\{1,\ldots, \lceil\frac{n}{2}\rceil\}$ and a partition $Y = (Y_1, \ldots, Y_m)$ of $\{\lceil\frac{n}{2}\rceil + 1, \ldots, n\}$ by setting \begin{align*}
X_k &= \{1,\ldots, \lceil\textstyle\frac{n}{2}\rceil\} \cap \supp(w_k);\\
Y_k &= \{\lceil\textstyle\frac{n}{2}\rceil + 1, \ldots, n\} \cap \supp(w_k).\end{align*} By definition of $w_k$ we see that $X_k$ is an interval $[\underline x_k, \overline x_k]$ where $\underline x_k$ and $\overline x_k$ are, respectively, the minimal and maximal elements of $X_k$ appearing in $\supp(w_k)$. Similarly we may write $Y_k = [\underline y_k, \overline y_k]$ for appropriate $\underline y_k$ and $\overline y_k$. For example, if $\lambda = (8,5)$, then $w_{\lambda}$ is $(1,13,2,12,3,11,4,10)(5,9,6,8,7)$ and we have $X_1 = \{1,2,3,4\} = [1,4]$, $X_2 = \{5,6,7\}$, $Y_1 = \{10,11,12,13\}$, $Y_2 = \{8,9\}$. Note also that $\sigma(w_1) = (1,10)(2,11)(3,12)(4,13)$, $\sigma(w_2) = (6,8)(7,9)$, $\tau(w_1) = (1,4)(2,3)(11,13)$ and $\tau(w_2) = (5,7)(8,9)$. We will see that $\tau(w_k)$ leaves the sets $X_k$ and $Y_k$ invariant, and $\sigma(w_k)$ interchanges, in an order-preserving way, nearly all, if not all, elements of $X_k$ and $Y_k$.

\begin{prop}
\label{maintypea} Let $\lambda$ be a maximal partition of $n$ and let $C$ be the corresponding conjugacy class of $\sym(n)$. The corresponding element $w_{\lambda}$ has maximal length in $C$, and $e(w_{\lambda}) = 0$.
\end{prop}

\begin{proof} Write $\lambda = (\lambda_1, \ldots, \lambda_m)$.
Set $w = w_{\lambda} = w_1\cdots w_m$ where $w_i$ is as given in Definition \ref{defkim}. By Theorem \ref{kimthm} $w$ has maximal length in $C$. For each $i$ set $\sigma_i = \sigma(w_i)$ and $\tau_i = \tau(w_i)$. Since the supports (in other words the sets of non-fixed points) of $\sigma$ and $\tau$ are subsets of the support of $w_i$, it is clear that both $\sigma_i$ and $\tau_i$ commute with both $\sigma_j$ and $\tau_j$ whenever $i \neq j$. Hence $\sigma = \sigma_1\cdots \sigma_m$ and $\tau = \tau_1\cdots \tau_m$ are involutions with the property that, by Lemma \ref{2.3}, $\sigma\tau = w$.
We must show that $N(\sigma) \cap N(\tau) = \emptyset$. This will imply by Equation~\eqref{leneq} that $e(w) = 0$. \\

Consider the cycle $w_k$ of $w$. Then $w_k = (a_{L+1}, a_{L+2}, \ldots, a_{L+\lambda_k})$ (setting $L = \sum_{j=1}^{k-1}\lambda_j$). This means, depending on the parity of $L$, that $w_k = (i, n+1-i, i+1, n-i, \ldots)$ or $(n+2-i, i, n+1-i, \ldots)$  for some $i \geq 1$. The support of $w_k$ is $X_k \cup Y_k$.\\

Let us consider $\tau_k = \tau(w_k)$. Now if $\lambda_k$ is even, we have $\tau_k = \prod_{i=1}^{(\lambda_k/2-1)}(a_{L+i},a_{L + \lambda_k-i})$. If $\lambda_k$ is odd then $\tau_k = \prod_{i=1}^{\lfloor\lambda_k/2\rfloor}(a_{L+i},a_{L + \lambda_k + 1 -i})$. In both cases $\tau_k$ is mapping odd terms of the sequence $(a_i)$ to odd terms and even terms to even terms. In particular, $\tau_k \in \sym(X_k) \times \sym(Y_k)$. Therefore \begin{equation}
\label{ntauk}N(\tau_k) \subseteq \{e_i - e_j: \underline{x}_k \leq i < j\leq \overline{x}_k \} \cup \{e_i - e_j: \underline{y}_k \leq i < j\leq \overline{y}_k\}.
\end{equation} (If $\lambda_k$ is odd then we have equality here and $\tau_k = g_{\underline{x}_k,\ldots, \overline{x}_k}g_{\underline{y}_k,\ldots, \overline{y}_k}$.)  \\

Next we look at $\sigma_k$. If $\lambda_k$ is even, then setting $\mu = \lfloor \frac{\lambda_k}{2}\rfloor$ we have
$$\sigma_k = \prod_{i=1}^{\mu}(a_{L+i},a_{L + \lambda_k + 1 -i}).$$ If $\lambda_k$ is odd then $\sigma_k = \prod_{i=2}^{\mu + 1}(a_{L+i},a_{L + \lambda_k + 2 -i})$. What happens this time is that $\sigma_k$ is the order preserving bijection between the highest $\mu$ elements of $X_k$ and the lowest $\mu$ elements of $Y_k$.
Therefore \begin{align}
 N(\sigma_k) &= \{e_i - e_j : \underline x_k \leq \overline x_k + 1 - \mu \leq i \leq \overline x_k < j < \underline y_k\} \nonumber \\ & \cup \{e_i - e_j : \overline x_k < i < \underline y_k \leq j \leq \underline y_k + \mu - 1 \leq \overline y_k\}.\label{nsigk}
\end{align}
Now for $\ell \neq k$, we have that $\sigma_\ell$ fixes all $i$ for $i \notin X_\ell \cup Y_\ell$ and interchanges various elements of $X_\ell$ and $Y_\ell$. Therefore $\sigma_\ell(N(\sigma_k)) = N(\sigma_k)$. So we may apply Lemma \ref{useful} to conclude that $N(\sigma) = \dot{\cup}_{k=1}^mN(\sigma_k)$. Similarly since $\tau_\ell$ fixes all $i$ for $i \notin X_\ell \cup Y_\ell$, we can deduce that $N(\tau) = \dot{\cup}_{k=1}^mN(\tau_k)$. Looking at Equations \eqref{ntauk} and \eqref{nsigk} it is clear that $N(\tau) \cap N(\sigma) = \emptyset$. Therefore by Equation \eqref{leneq} we see that $\ell(w) = \ell(\sigma) + \ell(\tau)$ and hence $e(w) = 0$, as required.
\end{proof}

\section{Maximal lengths in types $B_n$ and $D_n$}\label{bdl}

Throughout this section, $W$ is assumed to be a Coxeter group of
type $B_n$ containing $\hat W$, the canonical index 2 subgroup of
type $D_n$. We will view elements of $W$ as signed cycles.
A cycle is called negative if it has an odd number of minus signs above its entries, and positive otherwise.
The conjugacy classes of $W$ are parameterized by
signed cycle type. So for $X$ a subset of a conjugacy class of
$W$, this data may be encoded by
$$\lambda(X) = (\lambda_1, \ldots, \lambda_{\nu_X}; \lambda_{\nu_X
+ 1}, \ldots, \lambda_{z_X})$$ %
where in this expression $\nu_X$ is the number of negative cycles, $z_X$ is the total
number of cycles, and $\lambda_1 \leq \ldots \leq
\lambda_{\nu_X}$, respectively $\lambda_{\nu_X + 1} \leq \cdots
\leq \lambda_{z_X}$, are the lengths of the negative, respectively
positive, cycles of $X$. So any element of $X$ has $\lambda(X)$ as
its signed cycle type.
Our main aim in this section is to prove Theorem \ref{thm2} and the following.

\begin{thm} \label{thm1} Suppose $\hat W$ is a Coxeter group of type
$D_n$, and let $\hat C$ be a conjugacy class of $\hat W$. Set $C = \hat C_0 = 
\hat Cw_0$, where $w_0$ is the longest element of $W$, and assume
that \linebreak $\lambda(C) = (\lambda_1, \ldots, \lambda_{\nu_C}; \lambda_{\nu
+ 1}, \ldots, \lambda_{z_C})$. Then the maximal length in $\hat C$
is \vspace*{-5mm}
$$n^2 + z_C - 2\sum_{i=1}^{\nu_C-1} (\nu_C - i)\lambda_i.$$\end{thm}

Let $\Phi$ be the root system of $W$. We employ the usual
description of $\Phi$ (as given, for example in \cite{humphreys}).
So the positive long roots are $\droot^+=\{e_i \pm e_j: 1\leq i < j
\leq n\}$, the negative long roots are $\droot^- = -\droot^+$ and
$\droot = \droot^+ \cup \droot^-$. The short roots are $\short^+ =
\{e_i: 1\leq i \leq n\}$, $\short^- = -\short^+$ and $\short =
\short^+ \cup \short^-$. Finally the positive roots are $\Phi^+ =
\droot^+ \cup \short^+$, the negative roots are $\Phi^- = \droot^-
\cup \short^-$ and $\Phi = \Phi^+ \cup \Phi^-$. We note that the
set of positive roots for $\hat W$ is $\droot^+$. We recall our convention will be
that the action of a group element is on the left of the root, so
that for example $(\mi 1 \+ 3 \+
8)(e_1) = (\+ 1 \+ 3 \+ 8)(\mi 1)(e_1) = -e_3$.\\

For $w \in W$, we define the following two sets.
$$\Lambda(w) = \{ \alpha \in \droot^+: w(\alpha) \in \Phi^-\};$$
$$\Sigma(w) = \{\alpha \in \short^+: w(\alpha) \in \Phi^-\}.$$

Set $l_B(w) =  |\Lambda(w)| + |\Sigma(w)|$ and $l_D(w) =
|\Lambda(w)|$. By \cite{humphreys} $l_B(w)$ is the length of $w$
and, should $w \in \hat W$, then $l_D(w)$ is the length of $w$
viewed as an element of $\hat W$. We call $l_B(w)$ the $B$-length of
$w$ and $l_D(w)$ the $D$-length of $w$. Given
$w \in W$, let $\overline{w}$ be the corresponding element of
$\sym(n)$. So, for example, if $w = (\mi 1 \+ 3 \+ 8)$, then
$\overline{w} = (138)$.
Observe that for $w \in W$, by a slight abuse of notation, we can write
$$w = \overline{w}\left(\textstyle\prod_{e_i \in \Sigma(w)} (\mi i)\right).$$
Hence, in our above example, $(\mi 1 \+ 3 \+ 8) = (138)(\mi 1)$.\\

Later when we talk about excess in these groups, to avoid ambiguity we will use the notation $e_B(w)$ to mean the excess $e(w)$ when $w$ is viewed as an element of $W$, and $e_D(w)$ to mean the excess $e(w)$ when $w$ is viewed (where appropriate) as an element of $\hat W$. That is, for all $w$ in $W$ we define \begin{align*}
e_B(w) &=\min\{\ell_B(\sigma) + \ell_B(\tau) - \ell_B(w): \sigma, \tau \in W, w = \sigma\tau, \sigma^2 = \tau^2 = 1\};\\
e_D(w) &=\min\{\ell_D(\sigma) + \ell_D(\tau) - \ell_D(w): \sigma, \tau \in \hat W, w = \sigma\tau, \sigma^2 = \tau^2 = 1\}.
\end{align*}

 As noted earlier, conjugacy classes of $W$ are
parameterized by signed cycle type. So,
for example, if $W$ is of type $B_9$ and $C$ is the $W$-conjugacy
class of $w = (\+{1} \+{2})(\+{3} \mi{4} \+{5})(\mi{6} \+{7}
\+{8})(\+{9})$, then the signed cycle type $\lambda(C)$ of $C$ is $\lambda(C) = (3,3;1,2)$.
 In $\hat W$, conjugacy classes are also parameterized by signed
cycle type, with the exception that there are two classes for each
signed cycle type consisting only of even length, positive cycles.
(The length profiles in each pair of split classes are identical, because the classes are interchanged by the length-preserving graph automorphism.)

\begin{lemma}\label{lmin} Let $C$ be a conjugacy class of $W$, and $w \in
C$. Then $$|\Lambda(w)| \geq n - z_C + 2\sum_{i=1}^{\nu_C-1} (\nu_C -
i)\lambda_i.$$ Moreover $|\Sigma(w)| \geq \nu_C$.\end{lemma}

\begin{proof} Set $\nu = \nu_C$ and $z = z_C$. Write $w$ as a product of disjoint cycles, $w =
\sigma_1 \sigma_2 \cdots \sigma_z$, where $\sigma_1, \ldots,
\sigma_{\nu}$ are negative cycles and the remaining cycles are
positive. Also, order the negative cycles such that $i < j$ if and
only if the minimal element in $\supp(\overline{\sigma_i})$ is
smaller than the minimal element in $\supp(\overline{\sigma_j})$. Our
approach is to consider certain $\langle w \rangle$-orbits of
roots.\smallskip\\

Firstly, let $\sigma$ be a positive $k$-cycle of $w$ and consider
the orbits consisting of roots of the form $e_a - e_b$, for $a, b
\in \supp(\overline{\sigma})$ and $a \neq b$. Each such orbit has length $k$. There
are $2\binom{k}{2}$ roots of this form, and hence $k-1$ such orbits.
Let $c$ be the maximal element in $\supp(\overline{\sigma})$. Then each orbit contains
both $e_a - e_c$ and $e_c - e_b$ for some $a, b \in \supp(\overline{\sigma})$. Now
$e_a - a_c \in \Phi^+$ and $e_c - e_b \in \Phi^-$. Therefore each
orbit includes a transition from positive to negative (that is, a
positive root $\alpha$ for which $w(\alpha)$ is negative). Hence
each orbit contributes at least one root to $\Lambda(w)$. Therefore
each positive $k$-cycle
contributes at least $k-1$ roots to $\Lambda(w)$.\\

Next suppose $\sigma$ is a negative $k$-cycle of $w$. This time we
consider orbits consisting of roots of the form $\pm e_a \pm e_b$,
for $a, b \in \supp(\overline{\sigma})$ and $a \neq b$. Each such orbit has length
$2k$. There are $4\binom{k}{2}$ roots of this form, and hence
$k-1$ such orbits. Moreover if $\alpha$ lies in one of these
orbits, then $-\alpha$ lies in the same orbit. Thus again each
orbit includes a transition from positive to negative and hence
contributes at least one root to $\Lambda(w)$. Therefore each
negative $k$-cycle contributes at least $k-1$ roots to $\Lambda(w)$.\\

Now suppose $\sigma_i$ and $\sigma_j$ are negative cycles, with $i
< j$, and consider the union of all orbits consisting of roots of
the form $\pm e_a \pm e_b$, where $a \in \supp(\overline{\sigma_i})$ and $b \in
\supp(\overline{\sigma_j})$. Suppose $|\supp(\overline{\sigma_i})| = k$ and $|\supp(\overline{\sigma_j})| = \ell$. Let $c$
be minimal in $\supp(\overline{\sigma_i})$. Then every orbit contains some $\pm e_c
\pm e_b$ for some $b \in \supp(\overline{\sigma_j})$. For every root of the form
$e_c \pm e_b$, we have $w^k(e_c \pm e_b) = -e_c \pm e_{b'}$ and
$w^{2k}(e_c\pm e_{b'}) = e_c \pm e_{b''}$ for some $b', b'' \in
\supp(\overline{\sigma_j})$. Now $e_c \pm e_b$ and $e_c \pm e_{b''}$ are positive
roots, but $-e_c \pm e_{b'}$ is negative. Therefore in this orbit
or part of orbit there is at least one transition from positive to
negative. There are $2\ell$ roots of the form $e_c \pm e_b$, and
hence each pair $\sigma_i$, $\sigma_j$ of negative cycles with $i
< j$ contributes at least $2|\supp(\overline{\sigma_j})|$ roots to $\Lambda(w)$. For
example, letting $i$ range from 1 to $\nu-1$, we get a total of
$(\nu-1) \times 2|\supp(\overline{\sigma_{\nu}})|$ roots from pairs $\sigma_i$
and $\sigma_{\nu}$.\\

Combining these three observations and writing $k_i$ for $|\supp(\overline{\sigma_i})|$, we see that%
$$\Lambda(w) \geq \sum_{i=1}^z (k_i-1) + 2\sum_{i=2}^{\nu} (i-1)k_i.$$

Since $\{k_1, \ldots, k_{\nu}\} = \{\lambda_1, \ldots,
\lambda_{\nu}\}$, and $\lambda_1 \leq \lambda_2 \leq \cdots \leq
\lambda_{\nu}$, it is clear that

\begin{align*} \sum_{i=2}^{\nu} (i-1)k_i &= k_2 + 2k_3 + \cdots + (\nu-1)k_\nu \\ 
&\geq \lambda_{\nu-1} + 2\lambda_{\nu-2} + \cdots + (\nu - 1)\lambda_1 \\ &= \sum_{i=1}^{\nu-1} (\nu-i)\lambda_i.\end{align*}

Therefore $$|\Lambda(w)| \geq n - z + 2\sum_{i=1}^{\nu-1} (\nu -
i)\lambda_i.$$

It only remains to show that $|\Sigma(w)| \geq \nu$. This
trivially follows from the fact that there are $\nu$ negative
cycles and each negative cycle must contain at least one minus
sign. Therefore there are at least $\nu$ roots $e_a$ for which
$w(e_a) \in \Phi^-$. Thus $|\Sigma(w)| \geq \nu$ and the proof of
the lemma is complete.\end{proof}

Next, given a conjugacy class $C$ of $W$ we define a particular
element $u_C$ of $C$ (which will turn out to have minimal
$B$-length). Recall that the signed cycle type of $C$ is
$$\lambda(C) = (\lambda_1, \lambda_2, \ldots, \lambda_{\nu_C};
\lambda_{{\nu_C}+1}, \ldots, \lambda_{z_C}),$$ and write $\mu_i = n
- \sum_{j = 1}^{i}\lambda_j$ for $1 \leq i < z_C$. Set $\nu = \nu_C$
and $z = z_C$. Then define $u_C$ to be the following element of $C$.
\begin{align*}u_C = &
(\+{1}, \+{2}, \ldots, \+{\lambda_z})(\+{\mu_{z-1} + 1}, \ldots, \+{\mu_{z-2}})\cdots (\+{\mu_{\nu + 1}+1}, \+{\mu_{\nu+1} + 2},\ldots, \+{\mu_{\nu}})\cdot\\
&\cdot (\+{\mu_{\nu}+1}, \+{\mu_{\nu} + 2,}\ldots, \+{\mu_{\nu-1}-1},\mi{\mu_{\nu-1}})\cdots (\+{\mu_{1}+1}, \ldots, \+{n-1}, \mi n)
\end{align*}%

As an example, let $w = (\mi{1} \+{7} \mi{2} \mi{9})(\mi{3} \+{4}
\mi{6})(\+{5} \mi{8})$ and let $C$ be the conjugacy class of $w$ in
type $B_9$. Then $\lambda_C = (2,4;3)$, ${\nu_C}=2$, $z_C = 3$,
$\mu_1 = 7$ and $\mu_2 = 3$. This gives $u_C = (\+{1} \+{2} \+{3})(\+{4}\+{5} \+{6} \mi{7})(\+{8} \mi{9})$.

\begin{lemma} \label{uc} Suppose $w = u_C$ for some conjugacy class $C$ of $W$. Then $|\Sigma(w)| = \nu_C$ and
$|\Lambda(w)| = n - z_C + 2\sum_{i=1}^{\nu_C-1} (\nu_C -
i)\lambda_i$\end{lemma}
\begin{proof} Again set $z = z_C$ and $\nu = \nu_C$. The size of $\Sigma(w)$ is simply the number of
minus signs appearing in the expression for $w$. Here, $\Sigma(w)
= \{e_n, e_{\mu_1}, \ldots, e_{\mu_{{\nu}-1}}\}$ and $|\Sigma(w)|
= {\nu}$.\smallskip\\
To find $\Lambda(w)$, consider pairs $(i, j)$ with $1\leq i < j \leq
n$. Suppose first that $i$ and $j$ are in the same cycle of
$\overline{w}$. Then $e_i \notin \Sigma(w)$ because only the maximal
element of each negative cycle has a minus sign above it. If $j =
\mu_{k}$ for some $k$, or if $j =n$, then exactly one of $e_i + e_j
\in \Lambda(w)$ or $e_i - e_j \in \Lambda(w)$ occurs (depending
whether $k < \nu$). Otherwise, $e_i - e_j \notin \Lambda(w)$ and
$e_i + e_j \notin \Lambda(w)$. Hence a cycle $(\+{\mu_{{k+1}} + 1}, \cdots, \+{\mu_{k}-1}, \arb{\mu_{k}})$
contributes exactly $\lambda_{k+1}-1$ roots to $\Lambda(w)$.\\

Now suppose that $i$ and $j$ are in different cycles. Hence
$\overline{w}(i) < \overline{w}(j)$. It is a simple matter to
check that if $e_i \in \Sigma(w)$, then $\{e_i + e_j, e_i - e_j\}
\subseteq \Lambda(w)$, whereas if $e_i \notin \Sigma(w)$, then
$e_i - e_j$ and $e_i + e_j$ are not in $\Lambda(w)$. Therefore each
$i$ with $e_i \in \Sigma(w)$ contributes exactly $2(n-i)$
additional roots to $\Lambda(w)$, and no roots are contributed
when $e_i \notin \Sigma(w)$.\\

Therefore \begin{eqnarray*}%
|\Lambda(w)| &=& \sum_{k=1}^z (\lambda_{k+1}-1) + \sum_{k: e_k \in
\Sigma(w)} 2(n-k)\\
&=& (n-z) + 2\left((n-n) + (n-\mu_1) + (n-\mu_2) + \cdots +
(n-\mu_{\nu-1})\right)\\
&=& (n-z) + 2\sum_{i=1}^{\nu-1} \sum_{j=1}^i \lambda_j\\
&=& n-z + 2\sum_{i=1}^{\nu-1} (\nu-i)\lambda_i.\end{eqnarray*} %
Therefore $|\Lambda(w)| = n - z_C + 2\sum_{i=1}^{\nu_C-1} (\nu_C -
i)\lambda_i$ and $|\Sigma(w)| = \nu_C$.
\end{proof}

\begin{cor} \label{cor1} Let $C$ be a conjugacy class of $W$. Then the
minimal $B$-length in $C$ is $$n + \nu_C - z_C +
2\sum_{i=1}^{\nu_C-1} (\nu_C - i)\lambda_i.$$ If $w \in C$ has
minimal $B$-length, then $|\Lambda(w)| = n - z_C +
2\sum_{i=1}^{\nu_C-1} (\nu_C - i)\lambda_i$ and $|\Sigma(w)| =
\nu_C$. Moreover, $u_C$ is a representative of minimal $B$-length
in $C$.\end{cor}

In the next corollary the element $u_C^t$ is the element obtained
from $u_C$ by taking its shortest positive cycle (which in this
context will be the cycle $(\+{n}\; \+{n-1}\; \ldots \+{m})$ for some odd
$m$), and putting minus signs over $n$ and $n-1$. In other words
it is the conjugate of $u_C$ by $t = (\mi n)$. Conjugation by $(\mi n)$ is the length preserving automorphism of $\hat W$ induced by the graph automorphism of the Coxeter graph $D_n$.

\begin{cor} \label{cor2} Let $C$ be a conjugacy class of $W$. If $C$ is also a conjugacy class, or a union of conjugacy classes,
of $\hat W$, then the minimal $D$-length of elements in the
class(es) is $n - z_C + 2\sum_{i=1}^{\nu_C-1} (\nu_C - i)\lambda_i$.
Moreover $u_C$ and $u_C^{t}$ are representatives of minimal
$D$-length in the class(es), with one in each $\hat W$-class if
the class $C$ splits.\end{cor}

\begin{thm} \label{maxl} Let $C$ be a conjugacy class of $W$ and $w \in C$.
Let $C_0$ be the conjugacy class of $ww_0$ where $w_0$ is the
longest element of $W$. Then the maximal $B$-length of elements of
$C$ is $n^2 - |\Lambda(u_{C_0})| - |\Sigma(u_{C_0})|$, with $u_{C_0}w_0$ being an element of maximal $B$-length.
If $C$ is a conjugacy class or union of conjugacy classes
of $\hat W$, the maximal $D$-length of elements of $C$ is
$n^2 - n - |\Lambda(u_{C_0})|$. Moreover $u_{C_0}w_0$ and $u_{C_0}^tw_0$ are representatives of
maximal $D$-length in the class(es), with one in each $\hat W$-class if $C$ is a split class.\end{thm}

\begin{proof} Let $C$ be a conjugacy class of $W$. Since $w_0$
is central, the $W$-conjugacy class $C_0$ of $ww_0$ is just
$Cw_0$. Moreover, for any root $\alpha$ we have $w_0(\alpha) =
-\alpha$. Therefore for all $x \in W$, $|\Lambda(xw_0)| = (n^2-n) -
|\Lambda(x)|$ and $|\Sigma(xw_0)| = n - |\Sigma(x)|$. (Note that
there are $n^2-n$ long positive roots and $n$ short positive roots.)
Let $u = u_{C_0}$. Then by Lemmas \ref{lmin} and \ref{uc},
we have that for all $v \in C_0$, $|\Lambda(v)| \geq
|\Lambda(u)|$ and $|\Sigma(v)| \geq |\Sigma(u)|$. Now every $x \in
C$ is of the form $vw_0$ for some $v \in C_0$. Hence for every $x
\in C$, we have
\begin{eqnarray*} |\Lambda(x)| &\leq& n^2 - n - |\Lambda(u)|\; \; \; \mbox{
and}\\ |\Sigma(v)| &\leq& n- |\Sigma(u)|.\end{eqnarray*}%
Also $|\Lambda(uw_0)| = n^2-n - |\Lambda(u)|$ and $|\Sigma(uw_0)|
= n - |\Sigma(u)|$. Therefore the maximal $B$-length in $C$ is
$n^2 - n - |\Lambda(u)| + n- |\Sigma(u)| = n^2 - |\Lambda(u)| -
|\Sigma(u)|$ and this is attained by the element $uw_0$. Moreover,
if $C$ is a conjugacy class (or union of conjugacy classes) of
$\hat W$, then the maximal $D$-length is $n^2 - n - |\Lambda(u)|$
and this is attained by $uw_0$ (or $(uw_0)^t$ if the class
splits).\end{proof}

Theorem \ref{thm1} now follows immediately from Theorem \ref{maxl} and
Lemma \ref{uc}. All that remains in this section is to prove Theorem \ref{thm2}.

\paragraph{Proof of Theorem \ref{thm2}}
Note that each element $w_{\lambda,\rho}$, where $\lambda = (\lambda_1, \ldots, \lambda_m)$ is a maximal split partition with respect to $\rho$, is of the form $w_0u_C$ for some $u_C$. In particular we have $z_C = m$ and $\nu_C = m - \rho$. Thus each element $w_{\lambda,\rho}$ has maximal $B$-length and maximal $D$-length in the class $Cw_0$. Moreover given a class $C'$ of $W$, setting $C = C'w_0$ we see that $w_0u_{C}$ is $w_{\lambda,\rho}$ for some suitable $\lambda, \rho$, and so $w_{\lambda,\rho}$ is of maximal $B$-length and $D$-length in $C'$. \qed \\

It is the task of the next section to show that these elements $w_{\lambda,\rho}$ have excess zero.

%
%

\section{Excess zero in types $B_n$ and $D_n$}\label{bdex}

The aim of this section is to prove Theorem \ref{main} for $W$ and $\hat W$. In order to do this we will show that the elements $w_{\lambda,\rho}$ described in Theorem \ref{thm2} have excess zero both in $W$ and (if applicable) in $\hat W$. To obtain the required involutions $\sigma$ and $\tau$ such that $N(\sigma) \cap N(\tau) = \emptyset$ and $\sigma\tau = w$, we modify the definition of $g_{b_1, \ldots, b_k}$ given in Section \ref{secan}. We will only need to consider sequences of consecutive integers here though. Let $\{a+1, a+2, \ldots, a+k\}$ be a sequence of consecutive positive integers in $\{1,\ldots, n\}$. Define $g_{[a,k]}$ to be the permutation of $W$ that reverses the sequence and fixes all other $c \in \{1, \ldots, n\}$. (Essentially this is just $g_{b_1, \cdots, b_k}$ where $b_1 = a+1$, $b_2 = a+2$, $\ldots$, $b_k = a+k$, but viewed as an element of $W$ rather than $\sym(n)$.) Thus
 $g_{[a,k]}(a+i) = a + k + 1 - i$ for $1\leq i \leq k$. That is, $$g = (\+{a + 1},\+{a+k})(\+{a+2},\+{a+k-1}) \cdots (\+{a+\lfloor k/2 \rfloor}, \+{a +\lceil k/2 \rceil + 1}).$$ In particular, $g_{[a,k]}$ is an involution.\\

We also define $h_{[a,k]}$ to be $g_{[a,k]}$ with the plus signs replaced by minus signs. Thus $h_{[a,k]}(a+i) = -(a+k+1) + i$ for $1\leq i \leq k$. So $$h_{[a,k]} = \left\{\begin{array}{ll}
(\mi{a + 1},\mi{a+k})(\mi{a+2},\mi{a+k-1}) \cdots (\mi{a +\textstyle\frac{k}{2}}, \mi{a + \textstyle\frac{k}{2} + 1}) & \text{ if $k$ even};\\
(\mi{a + 1},\mi{a+k})(\mi{a+2},\mi{a+k-1}) \cdots (\mi{a + \lfloor \textstyle\frac{k}{2} \rfloor}, \mi{a + \lceil \textstyle\frac{k}{2} \rceil + 1}), (\mi{a + \lceil\textstyle\frac{k}{2}\rceil}) & \text{ if $k$ odd.}\end{array}\right.$$ In particular, $h_{[a,k]}$ is an involution. Moreover $h_{[a,k]}$ is order preserving on the intervals $[1,a]$, $[a+1,a+k]$ and $[a+k+1,n]$.\\

As an example $g_{[1,6]} = (\+ 2 \+ 7)(\+ 3 \+ 6)(\+ 4 \+5)$ and $h_{[3,5]} = (\mi 4 \mi 8)(\mi 5 \mi 7)(\mi 6)$.\\

Next we define two kinds of cycle and some involutions which are relevant to our analysis of the elements $w_{\lambda, \rho}$. Define \begin{align*}
w^-_{[a,k]} &= (\mi{a+1}, \mi{a+2}, \ldots, \mi{a+k-1}, \mi{a+k})\\
\sigma(w^-_{[a,k]}) &= h_{[a,k]}\\
\tau(w^-_{[a,k]}) &= g_{[a,k-1]}\\
w^+_{[a,k]} &= (\mi{a+1}, \mi{a+2}, \ldots, \mi{a+k-1}, \+{a+k})\\
\sigma(w^+_{[a,k]}) &= h_{[a+1,k-1]}\\
\tau(w^+_{[a,k]}) &= g_{[a,k]}
\end{align*}

\begin{lemma}
Let $w$ be either $w^-_{[a,k]}$ or $w^+_{[a,k]}$. Then writing $\sigma = \sigma(w)$ and $\tau = \tau(w)$ we have that $w = \sigma \tau$, where $\sigma$ and $\tau$ are both involutions.
\end{lemma}

\begin{proof} It is clear from the definitions that $\sigma$ and $\tau$ are involutions. First consider $w = w^-_{[a,k]}$. Then if $1\leq i \leq k-1$ we have $\sigma\tau(a+i) = \sigma(a+k-i) = -(a+k+1) + (k-i) = -(a + i+1)$, whereas $\sigma\tau(a+k) = \sigma(a+k) = -(a+k+1) + k = -(a+1)$. Therefore $w = \sigma\tau$ in this case.
Now consider $w = w^+_{[a,k]}$. If $1 \leq i \leq k-1$ we have $\sigma\tau(a+i) = \sigma(a+k+1-i) = \sigma((a+1) + (k-i))= -((a+1)+(k-1) + 1) + (k-i) = -(a + i + 1)$, whereas $\sigma\tau(a+k) = \sigma(a+1) = a+1$. Thus again $w=\sigma\tau$ and the proof is complete.\end{proof}

%


\begin{prop}
\label{maintypebd} Let $w = w_{\lambda,\rho}$ be the corresponding signed element of the maximal split partition $\lambda = (\lambda_1, \ldots, \lambda_m)$ (with respect to $\rho$). Then $e_B(w) = e_D(w) = 0$.
\end{prop}
\begin{proof} By definition, and recalling that $\mu_i = \sum_{j=1}^{i-1} \lambda_j$ we have $w = w_1 \cdots w_m$ where
$$w_i = \left\{\begin{array}{ll}
(\mi{\mu_i + 1}, \mi{\mu_i + 2}, \ldots, \mi{\mu_{i+1}-1}, \mi{\mu_{i+1}})
 & \text{ if $1\leq i \leq \rho$;}\\
(\mi{\mu_i + 1}, \mi{\mu_i + 2}, \ldots, \mi{\mu_{i+1}-1}, \+{\mu_{i+1}}) & \text{ if $ \rho < i \leq m$.}
\end{array}\right.$$
Therefore $$w_i = \left\{\begin{array}{ll}
 w^-_{[\mu_i,\lambda_i]}& \text{ if $1\leq i \leq \rho$;}\\
 w^+_{[\mu_i,\lambda_i]} & \text{ if $ \rho < i \leq m$.}
\end{array}\right.$$
For each $i$ set $\sigma_i = \sigma(w_i)$ and $\tau_i = \tau(w_i)$. Since the supports of $\overline{\sigma_i}$ and $\overline{\tau_i}$ are subsets of the support of $w_i$, it is clear that both $\sigma_i$ and $\tau_i$ commute with both $\sigma_j$ and $\tau_j$ whenever $i \neq j$. Hence $\sigma = \sigma_1\cdots \sigma_m$ and $\tau = \tau_1\cdots \tau_m$ are involutions with the property that $\sigma\tau = w$.
We must show that $N(\sigma) \cap N(\tau) = \emptyset$. \\

Consider a cycle $w_k$ of $w$. Then $\tau(w_k)$ is either $g_{[\mu_k,\lambda_k-1]}$ or $g_{[\mu_k,\lambda_k]}$. The action of $g$ is to reverse the order of the sequence $\mu_k + 1, \ldots, \mu_k + \lambda_k$, reverse the order of the sequence $-\mu_k, \ldots, -\mu_k - \lambda_k$ and fix all other integers. Hence \begin{equation} N(\tau(w_k)) \subseteq \{e_i - e_j: \mu_k < i < j \leq \mu_{k+1}\}\label{ntaukbd}.\end{equation}
On the other hand $\sigma(w_k)$ is either $h_{[\mu_k,\lambda_k]}$ or $h_{[\mu_k+1,\lambda_k-1]}$, so is of the form $h_{[a,b]}$ where $a \geq \mu_k$ and $a + b = \mu_{k+1}$.
We observe that \begin{equation}
\label{nsigmak1}
\Sigma(h_{[a,b]}) = \{e_{a+1}, \ldots, e_{a+b}\}\subseteq \{e_{\mu_k+1}, \ldots, e_{\mu_{k+1}}\}.\end{equation}
Recall that $h_{[a,b]}$ fixes $e_i$ for all $i \notin \{a + 1, \ldots, a+b\}$ and $h_{[a,b]}(e_{i}) = -e_{2a + b+1-i} $. From this we see that \begin{equation}
\Lambda(h_{[a,b]}) = \{e_{i} + e_j: a < i < j \leq a+b\}\cup \{e_i \pm e_j: a < i < a+b < j \leq n\}.
\end{equation} Therefore
\begin{equation}
\label{nsigmak2}
\Lambda(\sigma(w_k)) \subseteq \{e_{i} + e_j: \mu_k < i < j \leq \mu_{k+1}\}\cup \{e_i \pm e_j: \mu_k < i < \mu_{k+1} < j \leq n\}.
\end{equation}

For $\ell \neq k$, we note that $\sigma_{\ell}$ and $\tau_{\ell}$ fix all $i$ for $i \notin \{\mu_{\ell}+1,\ldots, \mu_{\ell+1}\}$. In particular they stabilize (setwise) the sets $\{1,\ldots, \mu_k\}$, $\{\mu_k+1,\ldots, \mu_{k+1}\}$ and $\{\mu_{k+1}+1, \ldots, n\}$. Therefore $\sigma_{\ell}(N(\sigma_k)) = N(\sigma_k)$. So we may apply Lemma \ref{useful} to conclude that $N(\sigma) = \dot{\cup}_{i=1}^mN(\sigma_k)$ and that $N(\tau) = \dot{\cup}_{i=1}^mN(\tau_k)$. \\

Equations \eqref{ntaukbd}, \eqref{nsigmak1} and \eqref{nsigmak2} now imply that $N(\tau) \cap N(\sigma) = \emptyset$. Therefore by Equation \eqref{leneq} we see that $\ell_B(w) = \ell_B(\sigma) + \ell_B(\tau)$ and therefore $e_B(w) = 0$. But also $N(\tau)\cap N(\sigma) = \emptyset$ implies that $\Lambda(\tau) \cap \Lambda(\sigma) = \emptyset$, and so we also have $\ell_D(w) = \ell_D(\sigma) + \ell_D(\tau)$, giving $e_D(w) = 0$ as required.
\end{proof}

\begin{cor} \label{why} Theorem \ref{main} holds for Coxeter groups of type $B_n$ and $D_n$.\label{bncor}\end{cor}

\begin{proof}
If $W$ is of type $B_n$, then by Theorem \ref{thm2} every conjugacy class $C$ of $W$ contains an element of the form $w_{\lambda,\rho}$ for suitable $\lambda$ and $\rho$, and this element has maximal $B$-length in $C$. By Proposition \ref{maintypebd} $w_{\lambda,\rho}$ has excess zero. Now consider $\hat W$ of type $D_n$, and let $C$ be a conjugacy class of $\hat W$. If $C$ is also a conjugacy class of $W$, then again $C$ contains some $w_{\lambda, \rho}$, which has maximal $D$-length and excess zero. If $C$ is not a conjugacy class of $W$ then $C \cup C^{(\mi n)}$ is a conjugacy class of $W$ (as conjugation by $(\mi n)$ is a length preserving map corresponding to the non-trivial graph automorphism of $D_n$), so for some $w = w_{\lambda, \rho}$ we have either $w$ or $w^{(\mi n)} \in C$. Now $e(w) = 0$, which means there are $\sigma, \tau$ involutions such that $w = \sigma\tau$ and $\ell(w) = \ell(\sigma) + \ell(\tau)$. Hence $w^{(\mi n)} = \sigma^{(\mi n)}\tau^{(\mi n)}$ and, since conjugation by $(\mi n)$ is a length-preserving map, we have $\ell(w^{(\mi n)}) = \ell(\sigma^{(\mi n)}) + \ell(\tau^{(\mi n)})$. Hence Corollary \ref{why} holds.\end{proof}

\section{Conclusion}\label{secend}

\paragraph{Proof of Theorem \ref{main}}
Observe that every finite Coxeter group $W$ is a direct product of irreducible Coxeter groups. If $W = W_1 \times \cdots \times W_n$ for some $W_i$, then it is easy to see that for $w = (w_1,\ldots, w_n) \in W$, we have $\ell(w) = \ell(w_1) + \cdots + \ell(w_n)$ and $e(w) = e(w_1) + \cdots + e(w_n)$. Moreover $w$ is of maximal length in some conjugacy class $C$ of $W$ if and only if each $w_i$ is of maximal length in a conjugacy class of $W_i$. Therefore Theorem \ref{main} holds if and only if it holds for all finite irreducible Coxeter groups. Theorem \ref{main} has already been proved for types $A_n$, $B_n$ and $D_n$ (Proposition \ref{maintypea} and Corollary \ref{bncor}). The exceptional groups $E_6$, $E_7$, $E_8$, $F_4$, $H_3$ and $H_4$ were checked using the computer algebra package {\sc Magma}\cite{magma}. In each case there is at least one (usually many) elements of maximal length and excess zero in every conjugacy class. Finally it is easy to check that every element of the dihedral group has excess zero, so the result is trivially true. Thus Theorem \ref{main} holds for every finite irreducible Coxeter group, and hence for all finite Coxeter groups.\qed \\

It is not the case that every element of maximal length in a conjugacy class always has excess zero. If $W$ is of type $E_6$, then in 23 of the 25 conjugacy classes every element of maximal length has excess zero. For the remaining two classes the situation is as follows. In the first class the elements have order 3. The maximal length of elements is 22, and there are 146 elements of maximal length. Of these, 134 have excess zero and the remaining 12 have excess 2. Using the standard generators $w_1, w_2, w_3, w_4, w_5, w_6$ for $E_6$, a representative of this class is $w_2w_3w_1w_4w_2w_3w_5w_4w_2w_3w_6w_5w_4w_3$. In the second class elements have order 6. The maximal length of elements is 20; there are 180 elements of maximal length of which 136 have excess zero and 44 have excess 2. A representative of this class is $w_5w_4w_2w_3w_1w_4w_3w_5w_6w_5w_4w_2w_3w_1$. If $W$ is of type $E_7$, then every element of maximal length in 59 of the 60 conjugacy classes has excess zero. In the remaining class,
  elements have order 3. The maximal length is 54, and there are 708 elements of maximal length, all but 50 of which have excess zero. A representative of this class is $w_1w_3w_1w_6w_5w_4w_2w_3w_1w_4w_3w_5w_4w_2w_6w_5w_7w_6w_5w_4w_2w_3w_1w_4w_3w_5$. \\

For the classical Weyl groups, we have checked all conjugacy classes of these groups with rank up to 10, and in each case every element of maximal length in a conjugacy class has excess zero. However there are examples of elements $w$ of maximal length with an arbitrarily large number of pairs of involutions $xy$ with $w = xy$, such that only one such pair has the property that $\ell(w) = \ell(x) + \ell(y)$. Elements like these `only just' have zero excess, see Lemma \ref{5.1} below. Because of this, and the examples in the exceptional groups, we are not sufficiently confident to conjecture that the pattern of maximal length elements having zero excess will continue, yet it might$\ldots$.

\begin{lemma}\label{5.1}
Let $W$ be of type $B_n$, for $n \geq 2$. There are at least $2^n$ pairs of involutions $(x,y)$ such that $xy = (\mi 1 \+ 2)$, but only one of these pairs has the property that $\ell(x) + \ell(y) = \ell((\mi 1 \+2))$.
\end{lemma}

\begin{proof} The element $w = (\mi 1 \+ 2)$ is certainly of maximal length in its conjugacy class, by Theorem \ref{thm2}.
If $x$ is an involution such that $xy = w$ for some involution $y$, then $w^x = w^{-1}$. Thus $x = x_1x_2$ where $x_1$ and $x_2$ are commuting involutions, $x_2$ fixes $1$ and $2$, and $x_1$ is either $(\mi 1)$, $(\mi 2)$, $(\mi 1 \mi 2)$ or $(\+ 1 \+ 2)$. We can then determine $y$, and the upshot is that we get the following possibilities, where here $z$ is any involution fixing $1$ and $2$.
\begin{center}
\begin{tabular}{cc}
$x$ & $y$\\
$(\mi 1)z$ &$(\mi 1 \mi 2)z$\\
 $(\mi 2)z$& $(\+ 1 \+ 2)z$\\
 $(\mi 1 \mi 2)z$& $(\mi 2)z$\\
 $(\+ 1 \+ 2)z$& $(\mi 1)z$
\end{tabular}
\end{center}
If $N(x) \cap N(y) = \emptyset$ then clearly we must have $z=1$. It is now a quick check to show that the only possibility is $x = (\mi 2)$, $y = (\+ 1 \+ 2)$. The number of possible pairs $(x,y)$ is four times the number of involutions in a Coxeter group of type $B_{n-2}$, which is
at least $2^{n-2}$, because for all subsets $\{a_1, \ldots, a_k\}$ of size $k$ of $\{3, 4, \ldots, n\}$ the element $(\mi a_1)\cdots (\mi a_k)$ is an involution.\end{proof}

\end{document}